\documentclass[12pt]{article}

\usepackage[utf8]{inputenc}
\usepackage[T1]{fontenc}
\usepackage{lmodern}
\usepackage[english]{babel}
\usepackage{amsmath,mathtools,amsthm,amssymb,amsfonts}
\usepackage{color}
\usepackage{comment}
\usepackage{hyperref}

\usepackage{mathrsfs}
\usepackage{graphicx}
\usepackage{float}
\usepackage{enumitem}
\usepackage{tikz} 
\usepackage{esint}
\usepackage{pgfplots}
\usepackage{esvect}
\usepackage{pdfpages}
\usepackage{nicefrac}
\usepackage{leftidx, tensor}
\usepackage{cleveref}
\usepackage{accents}
\usetikzlibrary{shapes,arrows}
\usetikzlibrary{calc}
\usetikzlibrary{shapes.geometric}
\usetikzlibrary{shapes.arrows}
\usetikzlibrary{arrows.meta}
\usetikzlibrary{decorations.markings}
\usetikzlibrary{fit}
\usetikzlibrary{patterns}
\usetikzlibrary{hobby}
\usepgfplotslibrary{patchplots}
\pgfplotsset{compat=1.11}

\numberwithin{equation}{section}

\renewcommand{\thefootnote}{\fnsymbol{footnote}}
\newcommand{\definedas}{\mathrel{\raise.095ex\hbox{\rm :}\mkern-5.2mu=}}

\newcommand{\R}{\mathbb{R}}

\newcommand{\Sbb}{\mathbb{S}}

\renewcommand{\d}{\,\mathrm{d}}

\newcommand{\ul}[1]{\underline{#1}}

\newcommand{\btr}[1]{\left\vert#1\right\vert}
\newcommand{\newbtr}[1]{\vert#1\vert}

\newcommand{\spann}[1]{\left\langle#1\right\rangle}

\newcommand{\Ric}{\mathrm{Ric}}
\newcommand{\scal}{\mathrm{R}}
\newcommand{\Riem}{\operatorname{Rm}}

\newcommand{\two}{\operatorname{II}}
\newcommand{\tr}{\text{tr}}
\newcommand{\dive}{\operatorname{div}}

\theoremstyle{plain}
\newtheorem{thm}{Theorem}[section]
\newtheorem*{thm*}{Main Theorem}
\newtheorem{prop}[thm]{Proposition}

\theoremstyle{definition}
\newtheorem{defi}[thm]{Definition}

\newtheorem{bem}[thm]{Remark}
\newtheorem{kor}[thm]{Corollary}

\begin{document}
	\begin{center}\LARGE A note on the stability of surfaces along null cones under area-preserving variations \end{center}
	\vspace{0.5cm}
	\begin{center}
		{\large Markus Wolff}
	\end{center}
	\vspace{0.4cm}
	\begin{abstract}
		In this note we investigate a notion of stability for spacelike cross sections of a null cone under area preserving variations that has been introduced in previous work by Kr\"oncke and the author. Here, we consider null cones with spherical cross sections in a $4$-dimensional spacetime and show that the Hawking energy of a stable cross section admits a non-negative lower bound provided the dominant energy condition holds. Similar to a recent work by Pe\~nuela Diaz, we show that under an additional assumption the Hakwing energy is zero if and only if the stable cross section embeds isometrically into the Minkowski lightcone. As a main result, we show that the only stable cross sections of the standard Minkowski lightcone are round spheres.  \\\\
	\end{abstract}
	
	\renewcommand{\thefootnote}{\arabic{footnote}}
	\setcounter{footnote}{0}
	
	\section{Introduction}\label{sec_intro}
	
		In Riemannian geometry, many fundamental notions of stability arise directly from a geometric variational problem. While the perspective from an underlying variational problem often breaks down when considered in a more general setup, analogous notions of stability have been studied with great success in the context of general relativity. In particular, various notions of stability have been considered in close analogy to the stability for minimal and constant mean curvature surfaces in Riemannian geometry. See for example \cite{alaeelesourdyau,anderssonmarssimon, anderssonmetzger, gallowaymendes1, gallowaymendes2, gallowayschoen, hawking, marsMOTS, penueladiaz} for a non-exhaustive list. 
		
		Given a spacelike codimension-$2$ surface $\Sigma$ in an ambient spacetime $(\overline{M},\overline{g})$, posssibly restricted to a given initial data set $(M,g,K)$ or a given null hypersurface $\mathcal{N}$, one can consider the \emph{null expansions} $\ul{\theta}$, $\theta$ with respect to a null frame $\{\ul{L},L\}$ of $\Gamma(T^\perp\Sigma)$ and the \emph{spacetime mean curvature} $\mathcal{H}^2$, given as the Lorentzian length of the codimension-$2$ mean curvature vector, as suitable generalizations to the concept of mean curvature for a hypersurface in a Riemannian manifold. In particular, marginally trapped surfaces, where either $\ul{\theta}$ or $\theta$ vanish identically, present themselves as the natural analogue to minimal surfaces. The stability of marginally outer trapped surfaces (MOTS) and generalizations of MOTS stability have been for example studied in \cite{alaeelesourdyau, anderssonmarssimon,anderssonmetzger, gallowaymendes1, gallowaymendes2,gallowayschoen, hawking, marsMOTS} in analogy to the stability of minimal surfaces. 
		
		Here, we study a notion of stability under area preserving variations for spacelike cross sections $\Sigma$ of a given null hypersurface $\mathcal{N}$. More precisely, we restrict our attention to null cones \cite{roesch}, i.e., the expansion of $\mathcal{N}$ is strictly positive with respect to a choice of null generator $\ul{L}$. We note that this directly implies that $\ul{\theta}>0$ for every spacelike cross section $\Sigma$ with respect to $\ul{L}$. In view of the typical toy models, such as the Minkowski and Schwarzschild lightcone, it is a natural assumption to make when not considering non-expanding horizons \cite{roesch}, i.e., models for black hole and cosmological horizons. We note however, that in general a null hypersurface is neither.
		
		We say a spacelike cross section $\Sigma$ of a null cone $\mathcal{N}$ is stable (under area-preserving) variations if 
		\[
		\int_\Sigma fJ(f)\ge 0
		\] 
		for all $f\in W^{1,2}(\Sigma)$ such that $\int_\Sigma f\d\mu=0$, where the \emph{Jacobi operator} $J$ is derived from the linearization of the spacetime mean curvature $\mathcal{H}^2$. See \cite{kroenckewolff}. We note that $J$ is gauge-invariant and therefore purely determined by the geometry of $\Sigma$. While the condition $\int_\Sigma f\d\mu=0$ is commonly associated to volume-preserving variations in Riemannian geometry, in particular in the context of CMC stability, there is a natural relation to area-preserving variations on null cones. In fact, we note that there is likely no well-defined notion of enclosed volume in the null case. In this sense, although not arising from an underlying variational problem, the above notion of stability presents itself as a natural null analogue to CMC stability in the Riemannian setting. See Section \ref{sec_stability} for details. We note further that the leaves of the asymptotic foliation by surfaces of constant spacetime mean curvature (STCMC) of an asymptotically Schwarzschildean lightcone constructed in \cite{kroenckewolff} are strictly stable in the above sense.
		
		Here, we also mention the work of Alaee, Lesourd and Yau \cite{alaeelesourdyau} and work of Pe\~nuela Diaz \cite{penueladiaz} who consider similar notions of stability for codimension-$2$ surfaces in an ambient spacetime. In fact, the notion of $\textbf{H}$-stability (under a volume-preserving condition) in \cite{alaeelesourdyau} is equivalent to the the notion of stability considered here if taken with respect to the null direction $\ul{\theta}^{-1}\ul{L}$.
		
		As a first result, we show that a stable spacelike cross section $\Sigma$ with $\mathcal{H}^2\ge0$ of a null cone $\mathcal{N}$ satisfies
		\[
			\int_\Sigma\mathcal{H}^2\le 16\pi
		\]
		if the ambient spacetime satisfies the dominant energy condition. In particular,
		\[
			m_H(\Sigma):=\sqrt{\frac{\btr{\Sigma}}{16\pi}}\left(1-\frac{1}{16\pi}\int_\Sigma\mathcal{H}^2\d\mu\right)\ge 0,
		\]
		where $m_H(\Sigma)$ is the \emph{Hawking energy} of $\Sigma$. See Proposition \ref{prop_hawkingenergy}. This can be understood as the null analogue to an estimate 
		by Christodoulou and Yau \cite{christodoulouyau} for stable CMC surfaces. Note that due to the equivalence to $\textbf{H}$-stability in a suitable null direction, this also directly follows from \cite[Theorem 3.3]{alaeelesourdyau}. See also \cite[Theorem 4.1]{penueladiaz} for a related result. We also provide a rigidity statement under an additional assumption in a similar spirit to \cite[Theorem 4.1]{penueladiaz}, see Proposition \ref{prop_rigidity}.
		
		As a main result (Corollary \ref{kor_mainresult}), we show that any stable spacelike cross section of the standard Minkowski lightcone is an STCMC surface, i.e., a round sphere with constant curvature. This can be understood as the analogue to the well-known fact that any stable topological $2$-sphere in $\R^3$ is a round sphere. The proof relies on a new set of test functions that are derived from a Minkowskian $4$-vector for spacelike cross sections introduced by the author in \cite{wolff4}.\\\\
		
		This paper is structured as follows: In Section \ref{sec_prelim} we briefly summarize some necessary prerequisites. In Section \ref{sec_stability} we introduce the Jacobi operator $J$ and define stability on a null cone. We further show that the Hawking mass is non-negative for stable spacelike cross sections provided the dominant energy condition holds and provide a rigidity statement under an additional assumption. In Section \ref{sec_classS} we look at the stability operator in a class of highly symmetric spacetimes and prove our main result. Finally, we derive another set of test functions motivated by Lorentz transformations in Section \ref{sec_Lorentztrafo}. We make some geometric observations and collect all the necessary identities for applications in the future.\\
		
		\begin{center}
			\textbf{\large Acknowledgments}\\
		\end{center}
		The author thanks Gregory J. Galloway for his interest and inspiring discussions. \newline This research 
		was funded in whole or in part by the Austrian Science Fund (FWF)\newline 10.55776/ESP1094725.

	\section{Preliminaries}\label{sec_prelim}
	
		In the following, $(\overline{M},\overline{g})$ will always denote a $4$-dimensional spacetime, i.e., a time-oriented Lorentzian manifold with signature $(-,+,+,+)$. In particular, we always assume that $(\overline{M},\overline{g})$ admits a globally defined, timelike vector field $X\in\Gamma(TM)$ which induces the well-known notions of future-directed respectively past- directed vectors, vector fields, and curves. We recall that $(\overline{M},\overline{g})$ is said to satisfy the \emph{dominant energy condition} (DEC) if
		\[
			G(Y,Z)\ge 0
		\]
		for all future-directed causal (timelike or null) vector fields $Y,Z\in\Gamma(T\overline{M})$, where the Einstein tensor $G$ is given by 
		\[
			{G}(Y,Z):=\overline{\Ric}(Y,Z)-\frac{1}{2}\overline{\scal}\,\overline{g}(Y,Z).
		\]
		Note that the dominant energy condition directly implies the null energy condition (NEC), i.e.,
		\[
			\overline{\Ric}(L,L)\ge 0
		\]
		for all null vector fields $L\in\Gamma(T\overline{M})$. Here, we use the following conventions for the Riemann curvature tensor $\Riem$, Ricci curvature $\Ric$, and scalar curvature $\scal$ of a semi-Riemannian manifold $(M,g)$:
		\begin{align*}
			\Riem(X,Y,W,Z)&=g(\nabla_X\nabla_YZ-\nabla_Y\nabla_XZ-\nabla_{[X,Y]}Z,W),\\
			\Ric(X,Y)&=\tr_g\Riem(X,\cdot,Y,\cdot),\\
			R&=\tr_g \Ric.
		\end{align*}
		
	We say an oriented hypersurface $\mathcal{N}$ in $(\overline{M},\overline{g})$ is a null hypersurface if its induced metric is degenerate. Equivalently, there exists a vector field $\ul{L}\in \Gamma(T\mathcal{N})$ such that $\ul{L}$ is orthogonal to all tangent directions, i.e.,
	\[
		T_p\mathcal{N}=(\ul{L}_p)^\perp\subseteq T_p\overline{M}\text{ for all }p\in\mathcal{N}.
	\]
	In particular, $\ul{L}$ is orthogonal to itself and therefore null. In the following, we briefly state the properties of null hypersurfaces relevant to our discussion without proof. For a more thorough introduction into null hypersurfaces we refer the interested reader to \cite[Section 4.7]{wolff_thesis} and the references given therein. First, we note that $\mathcal{N}$ is ruled by the integral curves of $\ul{L}$ and call $\ul{L}$ a null generator. The choice of null generator is not unique\footnote{Note that any two choices of null generators are related by a reparametrization of the integral curves, i.e., by scaling of a nowhere vanishing function.} and $\ul{L}$ can always be chosen such that its integral curves are null geodesics. Hence, a null hypersurface $\mathcal{N}$ is ruled by null geodesics. The null second fundamental form of $\mathcal{N}$ (with respect to $\ul{L})$ is defined as
	\[
		\ul{\chi}(X,Y):=-\overline{g}(\overline{\nabla}_XY,\ul{L})
	\]
	for $X,Y\in\Gamma(T\mathcal{N})$, where $\overline{\nabla}$ denotes the Levi-Civita connection of $(\overline{M},\overline{g})$. 
	
	Let $(\Sigma, \gamma)$ be an orientable spacelike codimension-$2$ surface in $(\overline{M},\overline{g})$. Recall that the \emph{vector-valued second fundamental form} of $\Sigma$ in $(M,g)$ is defined as
	\[
	\vec{\two}(V,W)=\left(\overline{\nabla}_VW\right)^\perp
	\]
	for all tangent vector fields $V, W\in\Gamma(T\Sigma)$. Further, the \emph{codimension-$2$ mean curvature vector} $\vec{\mathcal{H}}$ of $\Sigma$ is given by the trace of $\vec{\two}$ with respect to $\gamma$, i.e., $\vec{\mathcal{H}}=\tr_\gamma\vec{\two}$. Additionally, we define the \emph{spacetime mean curvature} $\mathcal{H}^2$ as the Lorentzian length of $\vec{\mathcal{H}}$, i.e., 
	\[
	\mathcal{H}^2=\overline{g}(\vec{\mathcal{H}},\vec{\mathcal{H}}).
	\]
	If $\vec{\mathcal{H}}$ is spacelike, this agrees with the notion of spacetime mean curvature by Ceder\-baum-- Sakovich \cite{cederbaumsakovich} upon taking a square root. However, as $\vec{\mathcal{H}}$ will in general have no  fixed causal character, $\mathcal{H}^2$ can be at least locally negative. Indeed, \emph{trapped surfaces}, where $\vec{\mathcal{H}}$ is timelike everywhere along $\Sigma$ arise naturally in the context of general relativity. Additionally, we call $\Sigma$ a surface of constant spacetime mean curvature (STCMC surface) if $\mathcal{H}^2$ is constant along $\Sigma$; see \cite{cederbaumsakovich}. In terms of the spacetime mean curvature, the Hawking energy $m_H(\Sigma)$ of $\Sigma$ is given by
	\[
		m_H(\Sigma)=\sqrt{\frac{\btr{\Sigma}}{16\pi}}\left(1-\frac{1}{16\pi}\int_\Sigma\mathcal{H}^2\d\mu\right).
	\]
	
	Let $\{\ul{L},L\}$ be a null frame of $\Gamma(T^\perp\Sigma)$ with $g(\ul{L},L)=2$.\footnote{Note that according to this convention, $\ul{L}$ and $L$ have a different causal character; if one is future-directed, the other is past-directed.} Then, $\vec{\operatorname{II}}$ and $\vec{\mathcal{H}}$ admit the decomposition 
	\begin{align}
		\begin{split}\label{eq_decomp1}
			\vec{\operatorname{II}}&=-\frac{1}{2}\chi\underline{L}-\frac{1}{2}\underline{\chi}L,\\
			\vec{\mathcal{H}}&=-\frac{1}{2}\theta\underline{L}-\frac{1}{2}\underline{\theta}L,
		\end{split}
	\end{align}
	where the null second fundamental forms $\underline{\chi}$ and $\chi$ with respect to $\underline{L}$ and $L$, respectively, are defined as
	\begin{align*}
		\underline{\chi}(V,W):=-\overline{g}\left( \overline{\nabla}_VW,\underline{L} \right)&\,&
		{\chi}(V,W)&:=-\overline{g}\left( \overline{\nabla}_VW,L \right),
	\end{align*}
	for tangent vector fields $V,W\in \Gamma(T\Sigma)$, and the null expansions $\underline{\theta}$ and $\theta$ with respect to $\underline{L}$ and $L$, respectively, as
	\begin{align*}
		\underline{\theta}:= \operatorname{tr}_\gamma\underline{\chi},&\,&
		{\theta}:= \operatorname{tr}_\gamma{\chi}.
	\end{align*}
	Under the conventions taken here, the (twice contracted) Gauss Equation becomes
	\begin{align}\label{eq_Gauss}
		\overline{\scal}-2\overline{\Ric}(\ul{L},L)+\frac{1}{2}\overline{\Riem}(\ul{L},L,L,\ul{L})=\scal-\mathcal{H}^2+\newbtr{\vec{\two}}^2,
	\end{align}
	where $\scal$ denotes the scalar curvature of $\Sigma$ and \eqref{eq_decomp1} implies that
	\begin{align}\label{eq_secondffnulldecomp}
		\newbtr{\vec{\two}}^2=\spann{\ul{\chi},\chi},\text{ and }\mathcal{H}^2=\underline{\theta}\theta.
	\end{align}
	See for example \cite[Equation (9)]{anderssonmetzger}, and see \cite[Proposition 4.24]{wolff_thesis} for a proof using the same conventions. We further recall that the connection-$1$ form $\zeta$ of $\Sigma$ (with respect to $\{\ul{L},L\}$) is defined as 
	\[
	\zeta(V):=\frac{1}{2}g\left(\overline{\nabla}_V\underline{L},L \right).
	\]
	In the following, we will always assume that $\Sigma$ is a spacelike cross section of a given null hypersurface $\mathcal{N}$, i.e., $\Sigma$ is a spacelike codimension-$2$ surface of $(\overline{M},\overline{g})$ such that $\Sigma\subseteq\mathcal{N}$ and all integral curves of the null generator $\ul{L}$ intersect $\Sigma$ exactly once. For a spacelike cross section $\Sigma$ of $\mathcal{N}$, we shall always take the normal null vector field $\ul{L}\in \Gamma(T^\perp\Sigma)$ to be the given choice of null generator $\ul{L}$. We note that a null generator $\ul{L}$ of $\mathcal{N}$ is orthogonal to any spacelike cross sections by definition. Additionally, we note that we may use $\ul{\chi}$ in this case both for the spacelike cross section $\Sigma$ and the null hypersurface $\mathcal{N}$ without ambiguity as 
	\[
		\ul{\chi}_\Sigma=\ul{\chi}_{\mathcal{N}}(\,\cdot\,\vert _{T\Sigma},\,\cdot\,\vert_{T\Sigma}).
	\] 
	In particular, we may extend $\ul{\theta}$ to a smooth function on $\mathcal{N}$ via a foliation of spacelike cross sections. In fact, $\ul{\theta}$ extends to a well-defined function on $\mathcal{N}$ that is independent of the choice of foliation; see \cite[Section 4.7 and 4.8]{wolff_thesis} for details. That is to say, the expansion $\ul{\theta}_\Sigma$ for any spacelike cross section of $\mathcal{N}$ is the restriction of $\ul{\theta}_\mathcal{N}$ to $\Sigma$, and again we simply use $\ul{\theta}$ by slight abuse of notation without ambiguity. In particular, $\ul{\theta}$ is only pointwise dependent on the spacelike cross section. We adopt the following definition by Roesch \cite[Definition 2.1]{roesch}:
	\begin{defi}
		We say $\mathcal{N}$ is a \emph{null cone} if there exists a null generator $\ul{L}$ such that $\ul{\theta}>0$ on $\mathcal{N}$.
	\end{defi}
	
	Examples of null cones are the (standard) lightcones in the Minkowski and Schwarzschild spacetime; see Section \ref{sec_classS} below. Note that for null hypersurfaces modeling black hole and cosmological horizons one has $\ul{\theta}=0$ throughout $\mathcal{N}$. In this sense, it is a natural restriction to assume that $\mathcal{N}$ is a null cone when one does not aim to model horizons. However, in general a null hypersurface is neither a null cone nor a horizon.
	
	Let $\Sigma$ be a spacelike cross section with $\ul{\theta}>0$. Then, the \emph{torsion} $\tau$ and \emph{scalar second fundamental form} $A$ are defined as
	\begin{align*}
		\tau&:=\zeta-\d\ln\ul{\theta},\\
		A&:=\ul{\theta}\chi,
	\end{align*}
	see \cite[Section 1]{brayroesch} and \cite[Section 3]{wolff1}. In particular, these objects are always well-defined for a spacelike cross section of a null cone. Moreover, it is easy to check that $\tau$, $A$, $\mathcal{H}^2=\tr A$, $\ul{\theta}^{-1}\ul{L}$, $\ul{\theta}^{-1}\ul{\chi}$ are well-defined and independent of the choice of null generator $\ul{L}$. That is, these well-defined objects on a null cone encode the geometry of $\Sigma$  and $\mathcal{N}$ independent of the gauge choice.
	
	\section{The stability operator}\label{sec_stability}
	
		Let $\mathcal{N}$ be a null cone, $\Sigma$ a spacelike cross section of $\mathcal{N}$, and let $x\colon (-\varepsilon,\varepsilon)\times\Sigma\to\mathcal{N}$ be a (local) smooth variation of $\Sigma$ along $\mathcal{N}$ with
		\[
			\frac{\d}{\d t}x=\varphi\ul{L}.
		\]
		By the well-known Raychaudhuri optical equations, see e.g. \cite[Proposition 4.27]{wolff_thesis}, we note that
		\begin{align*}
			\frac{\d}{\d t}\btr{\Sigma}=\int_{\Sigma}\ul{\theta}\varphi\d\mu,\\
			\frac{\d}{\d t}\mathcal{H}^2=Q(\varphi),
		\end{align*}
		where $Q$ is a second order linear operator which depends on the choice of null generator $\ul{L}$. In \cite[Section 3.2]{kroenckewolff} Kröncke and the author observed that 
		\[
			L(\varphi)=J(\ul{\theta}\varphi),
		\]
		where the second order elliptic operator $J$ is of divergence type and given by 
		\begin{align}
			\begin{split}\label{eq_Jacobi1}
			J(f)=&\,-2\Delta f-4\tau(\nabla f)-f\left[\mathcal{H}^2\left(1+\ul{\theta}^{-2}\left(\newbtr{\accentset{\circ}{\ul{\chi}}}^2+\overline{\Ric}(\ul{L},\ul{L})\right)\right)\right]\\
			&\,\,-f\left[\overline{\Ric}(\ul{L},L)-\frac{1}{2}\overline{\Riem}(\ul{L},L,L,\ul{L})+\spann{\ul{\theta}^{-1}\accentset{\circ}{\ul{\chi}},\accentset{\circ}{A}}+2\dive\tau+2\btr{\tau}^2\right].
			\end{split}
		\end{align}
		We call $J$ the Jacobi operator of $\Sigma$ and note that it is independent of the choice of null generator $\ul{L}$. From this perspective, it does seem convenient to rephrase the variation of a smooth family of spacelike cross sections $x\colon (-\varepsilon,\varepsilon)\times\Sigma\to\mathcal{N}$ along null cones as
		\[
		\frac{\d}{\d t}x=f\ul{\theta}^{-1}\ul{L}, 
		\]
		such that
		\begin{align*}
			\frac{\d}{\d t}\btr{\Sigma}=\int_{\Sigma}f\d\mu,\\
			\frac{\d}{\d t}\mathcal{H}^2=J(f).
		\end{align*}
		Exactly as in \cite[Definition 3.8]{kroenckewolff}, we now consider a notion of stability (under area-preserving variations) along null cones that can be understood as an analogue to the stability of CMC surfaces in the Riemannian setting (although we do not assume the STCMC condition here).
		\begin{defi}\label{defi_stability}
			Let $\mathcal{N}$ be a null cone, $\Sigma$ a spacelike cross section of $\mathcal{N}$. We say $\Sigma$ is stable (under area-preserving variations) if there exists a non-negative constant $c\ge 0$ such that
			\[
				\int_\Sigma fJ(f)\d\mu\ge c\int_\Sigma f^2\d\mu 
			\]
			for all $f\in W^{1,2}(\Sigma)$ such that $\int_\Sigma f\d\mu =0$. We say $\Sigma$ is strictly stable (under area-preserving variations), if $\Sigma$ is stable with $c>0$.
		\end{defi}
		Here, $ W^{1,2}(\Sigma)$ denotes a weighted Sobolev norm as defined in \cite[Section 2.2]{kroenckewolff}, but one may also consider the definition with unweighted norms.
		\begin{bem}\label{bem_stability}
			Observe that Equation \eqref{eq_secondffnulldecomp} and the Gauss Equation \eqref{eq_Gauss} yield that
			\[
				\overline{\Ric}(\ul{L},L)-\frac{1}{2}\overline{\Riem}(\ul{L},L,L,\ul{L})+\spann{\ul{\theta}^{-1}\accentset{\circ}{\ul{\chi}},\accentset{\circ}{A}}=\frac{1}{2}\mathcal{H}^2-\scal-{G}(\ul{L},L),
			\]
			where we recall  that $\scal$ denotes the scalar curvature of $\Sigma$ and ${G}=\overline{\Ric}-\frac{1}{2}\overline{R}\overline{g}$ denotes the Einstein tensor of $(\overline{M},\overline{g})$. As 
			\[
				\newbtr{\ul{\chi}}^2=\frac{1}{2}\ul{\theta}^2+\newbtr{\accentset{\circ}{\ul{\chi}}}^2,
			\]
			integration by parts yields that
			\begin{align*}
				\int_\Sigma fJ(f)\d\mu=\int_\Sigma 2\btr{\nabla f}^2-f^2\left[\mathcal{H}^2\left(1+\ul{\theta}^{-2}\left(\newbtr{\ul{\chi}}^2+\overline{\Ric}(\ul{L},\ul{L})\right)\right)-\scal-{G}(\ul{L},L)+2\btr{\tau}^2\right].
			\end{align*}
		\end{bem}
	\begin{bem}\label{bem_eigenvalues}
		Here, the author wants to comment on \cite[Remark 3.9]{kroenckewolff} where the (strict) stability as in Definition \ref{defi_stability} is stated to be connected to a lower bound on the first eigenvalue of $J$ via a Rayleigh quotient. Let us first comment that in general $J$ is not self-adjoint, so eigenvalues of $J$ have to be defined similar as in \cite{anderssonmarssimon} in analogy to the notion of MOTS stability. To circumvent this subtlety, one may consider the self-adjoint operator ${\widetilde{J}(f):=J(f)+4\tau(\nabla f)+2\dive\tau f}$ as
		\[
			\int_{\Sigma} fJ(f)\d\mu=\int_{\Sigma}f\widetilde{J}(f)\d\mu,
		\] 
		and where $\widetilde{J}=J$ if $\tau\equiv 0$, which is for example always satisfied in Class $\mathcal{S}$ (see Section \ref{sec_classS} below). Even though $\widetilde{J}$ now has a discrete, real spectrum
		\[
			\lambda_0\le \lambda_1\le \dotsc
		\]
		it is in general not clear whether $\lambda_0<\lambda_1$ or if constant functions are eigenfunctions with respect to $\lambda_0$. Hence, stability in the sense of Definition \ref{defi_stability} does in general not infer any information on $\lambda_1$. However, all of the above is true in the case of coordinate spheres $\Sbb^2_r$ in the Schwarzschild lightcone (of positive mass $m>0$) where constant functions are eigenfunctions with respect to the eigenvalue $\lambda_0$ and
		\[
			\lambda_0=-\frac{4}{r^2}+\frac{12m}{r^3}<\frac{12m}{r^3}=\lambda_1.
		\]
		Here we recall that in \cite{kroenckewolff}, the null hypersurface under consideration is assumed to be close to the Schwarzschild lightcone in a strong sense, see \cite[Definition 2.10]{kroenckewolff}, and spacelike cross sections are always assumed to be contained in a very restrictive a-priori class of surface $B_\sigma(B_1,B_2,B_3)$, see \cite[Definition 3.10]{kroenckewolff}, both of which enforce closeness to coordinate spheres in the Schwarzschild lightcone. Hence, while not immediately implied by the strict stability, it is still reasonable to expect that in this case the eigenfunctions with respect to $\lambda_0$ are close to constants and that
		\[
			\lambda_0\le -4+\frac{Cm}{\sigma^3}<0<\frac{\varepsilon m}{\sigma^3}\le \lambda_1
		\] 
		for suitable uniform constants $\varepsilon<12<C$ provided $\sigma$ is sufficiently large.
	\end{bem}
		
	In analogy to the work of Christodoulou--Yau \cite{christodoulouyau} and the recent work of Pe\~nuela Diaz \cite{penueladiaz}, we establish a lower bound on the Hawking energy of stable spacelike cross sections. We note that the result also directly follows from \cite[Theorem 3.3]{alaeelesourdyau} since $\textbf{H}$-stability (under a volume-preserving condition) in the sense of \cite{alaeelesourdyau} in direction $\ul{\theta}^{-1}\ul{L}$ is equivalent to stability as in Definition \ref{defi_stability}. In fact, we use the same set of test functions as in \cite{alaeelesourdyau, penueladiaz} but present a proof for the convenience of the reader.
	\begin{prop}\label{prop_hawkingenergy}
		Let $\mathcal{N}$ be a null cone with spherical cross sections in an ambient $4$-dimensional spacetime $(\overline{M},\overline{g})$ that satisfies the DEC, and let $\Sigma$ be a spacelike cross section with $\mathcal{H}^2\ge 0$. If $\Sigma$ is stable (with respect to area preserving variations) with constant $c\ge 0$, then
		\[
			m_{H}(\Sigma)\ge \frac{2}{3}c\left(\frac{\btr{\Sigma}}{16\pi}\right)^{\frac{3}{2}}.
		\]
	\end{prop}

	\begin{bem}\label{bem_hawkingenergy}
		If $c>0$, we note that the DEC may be replaced by a suitable smallness assumption on $\overline{\Ric}$. Indeed, for the asymptotic foliation by stable STCMC surfaces constructed by Kr\"oncke and the author in \cite{kroenckewolff}, the constant $c$ as established in \cite[Proposition 3.10]{kroenckewolff} and the decay assumptions on the curvature in \cite[Definition 2.10]{kroenckewolff} are such that one can still infer the positivity of the Hawking energy along the leaves of the foliation. Note that the positivity of the Hawking mass along the leaves of the foliation has independently been established by Penuela Diaz in \cite{penueladiaz} where he shows that the leaves of the foliation are also variationally stable, see \cite[Theorem 5.4]{penueladiaz}, in the sense of \cite[Definition 3.2]{penueladiaz}.
	\end{bem}
	\begin{proof}
		Note that as $-\ul{L}$ and $L$ are either both past- or future pointing, we have
		\[
			\overline{\Ric}(\ul{L},\ul{L})\ge 0,\text{ and }-{G}(\ul{L},L)={G}(-\ul{L},L)\ge 0
		\]
		by the dominant energy condition. Hence, by Remark \ref{bem_stability} , we find that
		\[
			\int_\Sigma \left[c+\frac{3}{2}\mathcal{H}^2\right]f^2\d\mu\le \int_\Sigma 2\btr{\nabla f}^2+\scal f^2\d\mu.
		\]
		for any $f\in W^{1,2}(\Sigma)$ with $\int_\Sigma f\d\mu=0$. Since $\Sigma$ is a topological $2$-sphere, there exists functions $\widetilde{f}_i$, $i=1,2,3$, such that
		\[
		\int_\Sigma \widetilde{f}_i\d\mu=0,\text{ }\sum\limits_{i=1}^2\widetilde{f}_i^2=1,\text{ }\sum\limits_{i=1}^3\int_\Sigma\btr{\nabla \widetilde{f}_i}^2\d\mu=8\pi,
		\]
		arising from a conformal mapping by Hersch's lemma \cite{hersch}. Hence, using these as test functions and summing over $i$, the above integral estimate and the Gauss--Bonnet theorem imply that
		\[
			\int_\Sigma\mathcal{H}^2\d\mu\le 16\pi-\frac{2}{3}c\btr{\Sigma}.
		\]
		We conclude that
		\[
			m_H(\Sigma)=\sqrt{\frac{\btr{\Sigma}}{16\pi}}\left(1-\frac{1}{16\pi}\int_\Sigma\mathcal{H}^2\d\mu\right)\ge \sqrt{\frac{\btr{\Sigma}}{16\pi}}\left(1-1+\frac{c}{24\pi}\btr{\Sigma}\right)=\frac{2}{3}c\left(\frac{\btr{\Sigma}}{16\pi}\right)^{\frac{3}{2}}.
		\]
	\end{proof}
	
	We can derive the following rigidity statement:
	
	\begin{prop}\label{prop_rigidity}
		Let $\mathcal{N}$ be a null cone with spherical cross sections in an ambient $4$-dimensional spacetime $(\overline{M},\overline{g})$ that satisfies the DEC, and let $\Sigma$ be a stable spacelike cross section with $\mathcal{H}^2\ge 0$. If 
		\[
		m_{H}(\Sigma)=0
		\]
		and $\overline{\Ric}(\ul{L},L)-\frac{1}{2}\Riem(\ul{L},L,L,\ul{L})$ does not change sign along $\Sigma$, then $\Sigma$ embeds isometrically into the standard lightcone of the Minkowski spacetime $(\R^{1,3},\eta)$ such that
		\[
			\eta(\vec{\mathcal{H}}_{Mink},\vec{\mathcal{H}}_{Mink})=\mathcal{H}^2=g(\vec{\mathcal{H}},\vec{\mathcal{H}}),
		\]
		where $\vec{\mathcal{H}}_{Mink}$ denotes the mean curvature vector of $\Sigma$ in $(\R^{1,3},\eta)$.
	\end{prop}
	\begin{bem}\label{bem_rigidity}
		We note that the assumption on the sign of $\overline{\Ric}(\ul{L},L)-\frac{1}{2}\Riem(\ul{L},L,L,\ul{L})$ is closely related to the assumption in the recent rigidity result by Pe\~nuela Diaz in \cite[Theorem 4.1]{penueladiaz}. While we essentially show that $\mathcal{N}$ agrees with the Minkowski lightcone along $\Sigma$, the characterization by Pe\~nuela Diaz derived from the codimension-$2$ notion of stability goes further. That is, Pe\~nuela Diaz shows that $\Sigma$ is a round sphere and that any spacelike hypersurface $\Omega$ in $\overline{M}$ with $\partial\Omega=\Sigma$ embeds isometrically into the Minkowski spacetime. By analogy, although beyond the scope of this paper, it is thus reasonable to expect that the assumptions in Proposition \ref{prop_rigidity} in fact imply that $\mathcal{N}$ is (a part of) the standard Minkowski lightcone. The fact that any stable cross section of the Minkowski lightcone is a round sphere is established in Corollary \ref{kor_mainresult} below. In particular, $\mathcal{H}^2$ is necessarily a strictly positive constant.
	\end{bem}
	
	\begin{proof}
		By assumption, $\Sigma$ is a topological $2$-sphere. Hence the uniformization theorem implies that its induced metric $\gamma$ is conformally round, i.e., $\gamma=\omega^2\widehat{\gamma}$ for some positive function $\omega\colon\Sbb^2\to(0,\infty)$ where $\widehat{\gamma}$ denotes the standard round metric.
		It is a well-known fact that $\Sigma$ then embeds isometrically into the (future-directed and past- directed) standard lightcone in the Minkowski spacetime $(\R^{1,3},\eta)$ via the identification $\Sigma=\{\pm t=r=\omega\}$. See for example \cite{wolff1}. See also Section \ref{sec_classS} below. By the Gauss Equation in the Minkowski spacetime, see e.g. \cite[Remark 1]{wolff1}, we have
		\[
			\eta(\vec{\mathcal{H}}_{Mink},\vec{\mathcal{H}}_{Mink})=2\scal,
		\]
		so it remains to show that $\mathcal{H}^2=2\scal$. By the Gauss Equation \eqref{eq_Gauss} in $(\overline{M},\overline{g})$ we have
		\[
			\mathcal{H}^2-2\scal=2{G}(\ul{L},L)+2\spann{\ul{\theta}^{-1}\accentset{\circ}{\ul{\chi}},\accentset{\circ}{A}}+2\overline{\Ric}(\ul{L},L)-\overline{\Riem}(\ul{L},L,L,\ul{L})
		\]
		Note that by Proposition \ref{prop_hawkingenergy}, $m_H(\Sigma)=0$ shows that stability necessarily holds with constant $c=0$. Then, using the same test functions as above a closer examination of the integration by parts in Remark \ref{bem_stability} yields the following chain of inequalities
		\begin{align*}
			\frac{3}{2}\int_\Sigma\mathcal{H}^2\d\mu\le \int_\Sigma \mathcal{H}^2\left[1+\ul{\theta}^{-2}\left(\newbtr{\ul{\chi}}^2+\overline{\Ric}(\ul{L},\ul{L})\right)\right]-{G}(\ul{L},L)+2\btr{\tau}^2\d\mu \le 24\pi.
		\end{align*}
		By assumption, equality now holds everywhere. In particular, we find that $\newbtr{\accentset{\circ}{\ul{\chi}}}=0$ and ${G}(\ul{L},L)=0$ along $\Sigma$ as we assume the DEC. The above Gauss Equation hence simplifies to 
		\[
			\mathcal{H}^2-2\scal=2\overline{\Ric}(\ul{L},L)-\overline{\Riem}(\ul{L},L,L,\ul{L}).
		\]
		Further, by assumption and the Gauss--Bonnet theorem integration now yields that
		\[
			\int_\Sigma2\overline{\Ric}(\ul{L},L)-\overline{\Riem}(\ul{L},L,L,\ul{L})\d\mu =0.
		\]
		However, $2\overline{\Ric}(\ul{L},L)-\overline{\Riem}(\ul{L},L,L,\ul{L})$ has a fixed sign by assumption, so it in fact must vanish identically. Hence, $\mathcal{H}^2=2\scal$.
	\end{proof}
	
	\section{The stability operator in Class $\mathcal{S}$}\label{sec_classS}
	
	We now consider a class $\mathcal{S}$ of spacetimes of the form $\overline{M}=\R\times I\times \Sbb^2$ and 
	\[
		\overline{g}=-h(r)\d t^2+\frac{1}{h(r)}\d r^2+r^2\widehat{\gamma},
	\]
	where $h\colon (0,\infty)\to\R$ is a smooth function and $I$ is an open interval on which $\btr{h}$ is strictly positive. Spacetimes of class $\mathcal{S}$ have been widely studied and although they are highly rigid they contain many physically relevant toy models such as the Schwarzschild spacetime corresponding to $h(r)=1-\frac{2m}{r}$. See for example \cite{cederbaumwolff} and the references given therein.
	
	We observe that any zeros $r_i$ of $h$ correspond to a Killing horizon in a suitable spacetime extension. If the Killing horizons are non-degenerate, here $h'(r_i)\not=0$, it is a well-known fact that the spacetime admits an extension in spirit of the Kruskal--Szekeres extension of the Schwarzschild spacetime. See e.g. \cite{brillhayward, cederbaumwolff,schindagui}. In a recent work \cite{cederbaumwolff}, Cederbaum and the author have revisited the construction which naturally gives rise to (global) double null coordinates $(u,v)$. For $v\not=0$, $u\not=0$, the level sets are null cones, see e.g. \cite{wolff6}, also referred to as the principal null hypersurfaces, and the level sets $u=0$, $v=0$ corresponds to Killing horizons. We refer to \cite{wolff6} where the following has been worked out in detail. Here, we mainly state the relevant facts that hold for any principal null hypersurface $\mathcal{N}=\{v=v_0\not=0\}$ or $\mathcal{N}=\{u=u_0\not=0\}$:
	
	There exists a choice of geodesic null generator $\ul{L}$ of $\mathcal{N}$ such that any spacelike cross section $\Sigma$ can be identified with a graph over the radial coordinate, i.e., $\Sigma=\{r=\omega\}\cap \mathcal{N}$ with $\omega\colon\Sbb^2\to(0,\infty)$. Note that $\omega$ is simultaneously the conformal factor of the induced metric $\gamma=\omega^2\widehat{\gamma}$. Moreover, for any spacelike cross section $\Sigma$ of $\mathcal{N}$ we have
	\begin{align*}
		\ul{\chi}&=\omega\widehat{\gamma},\\
		\ul{\theta}&=\frac{2}{\omega}>0,\\
		\tau&= 0,\\
		\mathcal{H}^2&=2\scal-\frac{4(1-h(\omega))}{\omega^2}.
	\end{align*}
	We leave it as an exercise to the reader that in this setting the Jacobi operator for a spacelike cross section of $\mathcal{N}$ becomes
	\[
		J(f)=-2\Delta f-\left(\mathcal{H}^2-\frac{2}{\omega}h'(\omega)\right)f
	\]
	by direct computation. We note that all of the above identities are also well-known for the Minkowski lightcone, see e.g. \cite{wolff1}, corresponding to the choice $h\equiv 1$.
	
	Using Hersch's lemma \cite{hersch} as before as well as the above identity for $J$ and $\mathcal{H}^2$ as above, stability in the sense of Definition \ref{defi_stability} (with $c=0$) now implies 
	\[
		0\le \int_\Sigma \frac{(1-h(\omega))}{\omega^2}+\frac{2h'(\omega)}{\omega}\d\mu,
	\]
	i.e., any stable spacelike cross section $\Sigma$ must lie in a region of $\mathcal{N}$ where the above function with respect to $h$ integrates to a non-negative number. While this rules out stable cross sections in the Schwarzschild lightcone of negative mass $m<0$, the above identity is trivially satisfies in the Minkowski lightcone and the Schwarzschild lightcone of positive mass $m>0$. That is, there is a plethora of meaningful choices for $h$ for which the usual trick via Hersch's lemma does not yield any geometric information on stable spacelike cross sections at all.
	
	We recall that the conformal mapping chosen in Hersch's lemma can also be motivated from a notion of conformal center of mass for topological $2$-spheres introduced by Onofri \cite{onofri}. Essentially equivalent to the choice of conformal mapping in \cite{hersch}, Onofri \cite{onofri} showed that one can always choose a suitable conformal diffeomorphism on $\Sbb^2$, i.e., a suitable M\"obius transformation, such that there exists a conformal factor $w$ satisfying
	\[
		\int_{\Sbb^2}w^2f_i\d\widehat{\mu}=0
	\]
	for all $i=1,2,3$ where the $f_i$ denote (a choice of) first spherical harmonics and $\d\widehat{\mu}$ is the area element of the standard round sphere $\widehat{\gamma}$. Indeed, the well-known formulas 
	\[
		\sum\limits_{i=1}^3f_i^2=1,\text{ }\sum\limits_{i=1}^r\btr{\widehat{\nabla}f_i}_{\widehat{\gamma}}=2
	\]
	directly imply the properties of the functions in Hersch's lemma.
	
	In a similar spirit, the author showed in \cite{wolff4} that one can always choose a M\"obius transformation such that the conformal factor $w$ with respect to these coordinates satisfies the balancing condition
	\begin{align}\label{eq_balanced}
		\int_{\Sbb^2}w^3f_i\d\widehat{\mu}=0
	\end{align}
	for (a choice of) first spherical harmonics $f_i$. The balancing condition \eqref{eq_balanced} is derived from a timelike $4$-vector defined for spacelike cross sections of the Minkowski lightcone. Projecting this Minkowskian $4$-vector to the hyperboloid (of radius $1$), one can uniquely associate a vector $\vec{a}\in\R^3$ to a given conformally round metric $\omega^2\widehat{\gamma}$. See also \cite{cederbaumcortiersakovich}. As this Minkowskian $4$-vector is shown to transform equivariantly under Lorentz transformations in $\operatorname{SO}^+(1,3)$, which is isomorphic to the M\"obius group, the balancing condition \eqref{eq_balanced} is then always satisfied by
	\begin{align}\label{eq_conformalW}
		w(\vec{x})=\frac{\omega\circ\Phi_{-\vec{a}}(\vec{x})}{\sqrt{1+\btr{\vec{a}}^2}+\vec{a}\cdot\vec{x}}.
	\end{align}
	Here $\Phi_{-\vec{a}}$ is the M\"obius transformation uniquely determined by the Lorentz transformation in direction $-\vec{a}$ via the isomorphism of $\operatorname{SO}^+(1,3)$ and the M\"obius group. Morover, $\Phi_{-\vec{a}}$ also induces the corresponding change of coordinates, i.e., $\Phi_{-\vec{a}}$ is the conformal diffeomorphism such that $\gamma=\omega^2\widehat{\gamma}$ becomes $\gamma=w^2\widehat{\gamma}$ after this change of coordinates. We refer the interested reader to \cite[Section 3]{wolff4}. Here, we merely stress again that while the Minkowskian $4$-vector is understood from an extrinsic perspective in \cite{wolff4}, the balancing condition \eqref{eq_balanced} can also be derived from a purely intrinsic perspective for a conformally round surface. Thus, although Lorentz transformations are in general not isometries of the ambient spacetime of class $\mathcal{S}$, for a given spacelike cross section $\Sigma$ in $\mathcal{N}$ as above we can always choose coordinates that the induced metric satisfies $\gamma=w^2\widehat{\gamma}$ with $w$ satisfying \eqref{eq_balanced}. While this \emph{balanced} conformal factor no longer corresponds to the graph function $\omega$, we note that \eqref{eq_conformalW} yields that
	\begin{align}\label{eq_conformalOmega}
		\omega(\vec{x})=\frac{w\circ\Phi_{\vec{a}}(\vec{x})}{\sqrt{1+\btr{\vec{a}}^2}-\vec{a}\cdot\vec{x}}
	\end{align}
	with respect to the balanced coordinates on $\Sigma \cong \Sbb^2$. As a consequence, the balancing condition \eqref{eq_balanced} yields new test functions for stable surfaces:
	
	\begin{prop}\label{prop_gradientestimate}
		Let $\mathcal{N}$ be a principal null hypersurface in a spacetime of class $\mathcal{S}$. Let $\Sigma$ be a stable spacelike cross section in $\mathcal{N}$ with graph function $\omega$. Then
		\[
			\int_{\Sbb^2}\btr{\widehat{\nabla}w}_{\widehat{\gamma}}^2\d\widehat{\mu}\le \int_{\Sbb^2}\left(\frac{(1-h(\omega))}{\omega^2}+\frac{h'(\omega)}{2\omega}\right)w^4\d\widehat{\mu}
		\]
		for any conformal factor $w$ satisfying \eqref{eq_balanced}.
	\end{prop}
	\begin{bem}\label{bem_gradientestimate}
		Since the underlying Minkowskian center transforms equivariantly under Lorentz transformations, we note that any two conformal factors $w_1$, $w_2$ satisfying \eqref{eq_balanced} are related by an isometry on $\Sbb^2$, i.e., a rotation.
	\end{bem}
	\begin{proof}
		Using the explicit formulas for $\mathcal{H}^2$ and $J$ as above, a spacelike cross section $\Sigma$ is stable in $\mathcal{N}$ if and only if
		\[
			\int_\Sigma \scal f^2\d\mu\le \int_\Sigma \btr{\nabla f}^2+\left(\frac{2(1-h(\omega))}{\omega^2}+\frac{h'(\omega)}{\omega}\right)f^2\d\mu
		\]
		for all $f\in W^{1,2}(\Sigma)$ with $\int_\Sigma f\d\mu=0$. Using a M\"obius transformation as above, we may shift to coordinates such that $\gamma=w^2\widehat{\gamma}$ where the conformal factor $w$ satisfies the balancing condition \eqref{eq_balanced} and where the graph function $\omega$ on $\Sigma\cong\Sbb^2$ is now expressed via \eqref{eq_conformalOmega}. In particular,
		\[
			\int_{\Sigma}wf_i\d\mu=0
		\]
		for first spherical harmonics $f_i$, $i=1,2,3$. We recall that
		\[
			\sum\limits_{i=1}^3f_i^2=1,\text{ and }\sum\limits_{i=1}^3\btr{\widehat{\nabla}f_i}_{\widehat{\gamma}}^2=2,
		\]
		and we moreover observe that
		\[
			\sum\limits_{i=1}^3f_i\widehat{\gamma}(\widehat{\nabla} f_i,\widehat{\nabla}w)=0
		\]
		for any $C^1$ function $w$.\footnote{An easy way to verify this directly is in angular coordinates $(\theta,\varphi)\in(0,\pi)\times(0,2\pi)$ where $\widehat{\gamma}=\d\theta^2+\sin^2\theta^2\d\varphi^2$, $f_1=\sin\theta\cos\varphi$, $f_2=\sin\theta\sin\varphi$, $f_3=\cos\theta$.} Observe further that
		\[
			\int_\Sigma\btr{\nabla f}^2\d\mu=\int_{\Sbb^2}\btr{\widehat{\nabla}f}_{\widehat{\gamma}}^2\d\widehat{\mu}
		\]
		for any $C^1$ function $f$. Hence, using $wf_i$ as test functions for the stability and summing over $i$, the above identities and the conformal change of the area element yield that
		\[
			\int_\Sigma \scal w^2\d\mu\le \int_\Sigma 2+\btr{\nabla w}^2+\left(\frac{2(1-h(\omega))}{\omega^2}+\frac{h'(\omega)}{\omega}\right)w^2\d\mu.
		\]
		As $\gamma=w^2\widehat{\gamma}$, the well-known formula for the scalar curvature yields
		\[
			w^2\scal=2+2\btr{\nabla w}^2-w\Delta w.
		\]
		Hence, integration by parts and simplifying on both sides yields
		\[
			2\int_\Sigma\btr{\nabla w}^2\d\mu \le \int_\Sigma \left(\frac{2(1-h(\omega))}{\omega^2}+\frac{h'(\omega)}{\omega}\right)w^2\d\mu.
		\]
		Invoking the conformal change of the gradient and area element once again, the claim follows.
	\end{proof}
	
	The main application of the $L^2$-bound in Proposition \ref{prop_gradientestimate} here is a characterization of stable spacelike cross sections in the Minkowski lightcone.
	
	\begin{kor}\label{kor_mainresult}
		Let $\mathcal{N}$ be the standard lightcone in the Minkowski spacetime $\R^{1,3}$, $\Sigma$ a stable spacelike cross section in $\mathcal{N}$. Then $\Sigma$ is an STCMC surface, i.e., a round sphere with constant scalar curvature.
	\end{kor}
	\begin{bem}\label{main_result}
		Note that all STCMC surfaces in the Minkowski lightcone have been characterized as round spheres by Wang \cite{wang}. See also \cite{chenwang} and a recent proof by Palomo--Romero \cite{palomoromero}. It is easy to verify that all round spheres are stable and Corollary \ref{kor_mainresult} can be thus understood as an equivalence statement.
	\end{bem}
	\begin{proof}
		Since $\R^{1,3}$ corresponds to $h=1$, we have $\mathcal{H}^2=2\scal$ and Proposition \ref{prop_gradientestimate} implies that the balanced conformal factor $w$ is constant, i.e, $w=\rho>0$. Then, \eqref{eq_conformalOmega} implies that
		\[
			\omega(\vec{x})=\frac{\rho}{\sqrt{1+\btr{\vec{a}}^2}-\vec{a}\cdot\vec{x}}.
		\]
		In particular, $\gamma=\omega^2\widehat{\gamma}$ has constant scalar curvature, see e.g. \cite[Remark 2]{wolff1}, i.e., $\Sigma$ is a round sphere and STCMC surface in $\mathcal{N}$.
	\end{proof}
	
	\section{An explicit variation derived from Lorentz transformations}\label{sec_Lorentztrafo}
	
	The previous sections have already highlighted the significance of conformal diffeomorphisms when discussing the stability of topological $2$-spheres. Since any topological $2$-sphere is conformally round by the uniformization theorem, it admits an isometric embedding into the standard Minkowski lightcone as demonstrated above. In particular, we may utilize Lorentz transformations to generate a canonical variation along the Minkowski lightcone. While we give no immediate application of this here, we collect all the relevant identities for use in the future and observe some interesting geometric properties. Again, we emphasize that this extrinsic variation is equivalent to an intrinsic change of coordinates so applications are not limited to the Minkowski lightcone.
	
	Let $\Sigma$ be a spacelike cross section of the (future-/ past- directed) Minkowski lightcone, where we recall that without loss of generality we may always assume that ${\Sigma=\{\pm t=r=\omega\}}$ and the induced metric $\gamma$ of $\Sigma$ satisfies $\gamma=\omega^2\widehat{\gamma}$ with respect to the graph function\linebreak ${\omega\colon\Sbb^2\to(0,\infty)}$. Let $\vec{a}\in \R^3$ be a unit vector and consider a family of Lorentz transformations $\{L_{\vec{a},t}\}_{t\in(-1,1)}$ in direction $t\vec{a}$. Since the lightcone itself is invariant under Lorentz transformations in $\operatorname{SO}^+(1,3)$, this yields a smooth family of spacelike cross sections $\Sigma_{\vec{a},t}:=L_{\vec{a},t}(\Sigma)$. Note that $\Sigma=\Sigma_{\vec{a},0}$ and that $\Sigma_{\vec{a},t}=\Sigma_{\omega_{\vec{a},t}}$ with
	\[
	\omega_{\vec{a},t}=\frac{\omega\circ\Phi_{\vec{a},t}}{\sqrt{1+t^2}-t\vec{a}\cdot \vec{x}},
	\]
	where $\Phi_{\vec{a},t}$ is a uniquely determined M\"obius transformation on $\Sbb^2$, see e.g. \cite[Section 3]{wolff4}. Although the precise form of the variation in $t$
	\[
	\frac{\d}{\d t}\omega_{\vec{a},t}=\frac{\widehat{\gamma}(\widehat{\nabla}\omega,\frac{\d}{\d t}\Phi_{\vec{a},t})}{\sqrt{1+t^2}-t\vec{a}\cdot\vec{x}}-\frac{\omega_{\vec{a},t}}{\sqrt{1+t^2}-t\vec{a}\cdot\vec{x}}\left(\frac{t}{\sqrt{1+t^2}}-\vec{a}\cdot\vec{x}\right)
	\]
	is not very instructive at first sight, we recall the well-known facts that the area is preserved and the scalar curvature changes via a conformal diffeomorphism, i.e., 
	\begin{align*}
		\btr{\Sigma_{\vec{a},t}}&=\btr{\Sigma},\\
		\scal_{\vec{a},t}&=\scal_{\vec{a},0}\circ \Phi_{\vec{a},t}.
	\end{align*}
	Again, see for example \cite[Section 3]{wolff4}. We now define 
	\[
		f_{\vec{a}}:=\frac{2}{\omega}\frac{\d}{\d t}\omega_{\vec{a},t}\vert_{t=0}=2\widehat{\gamma}(\widehat{\nabla}\ln\omega,\frac{\d}{\d t}\Phi_{\vec{a},t}\vert_{t=0})+2 \vec{a}\cdot\vec{x}.
	\]
	Since we have implicitly chosen the null generator $\ul{L}$ such that $\ul{\theta}=\frac{2}{\omega}$, and $\mathcal{H}^2=2\scal$ by the Gauss equation in Minkowski, the variation formulae at the beginning of Section \ref{sec_stability} and the above identities for the area and scalar curvature along the variation yield that
	\begin{align}\label{eq_explicit_fa_1}
		\int_\Sigma f_{\vec{a}}\d\mu=0
	\end{align}
	and
	\begin{align}\label{eq_explicit_fa_2}
		-2\Delta f_{\vec{a}}-2\scal f_{\vec{a}}=J(f_{\vec{a}})=\frac{\d}{\d t}\mathcal{H}^2_{\vec{a},t}=2\widehat{\gamma}(\widehat{\nabla}R_\Sigma,\frac{\d}{\d t}\Phi_{\vec{a},t}\vert_{t=0}).
	\end{align}
	We are ready to state our first observation:
	\begin{prop}
		Let $(\Sigma,\gamma)$ be a topological $2$-sphere. For all $\vec{a}\in\Sbb^2$, we have
		\[
			\int_{\Sigma}f_{\vec{a}}\d\mu=\int_{\Sigma}\scal\cdot f_{\vec{a}}\d\mu=0.
		\]
	\end{prop}
	\begin{bem}
		Recall the well-known first variation formula for a hypersurface $N$ with mean curvature vector $\vec{H}$ in a Riemannian manifold $(M,g)$
		\[
			\int_N\dive_NX\d\mu_N+\int_Ng(\vec{H},X)\d\mu_N=0
		\]
		for all smooth, compactly supported vector fields $X$ on $M$. Taking $(M,g)=(\R^3,\delta)$ and $X=\vec{a}\in\Sbb^2$ in a neighborhood of a closed surface $N$, the fact that $\vec{a}$ is a parallel vectorfield and $\vec{H}=-H\nu$ yield that
		\[
			\int_NH\spann{\vec{a},\nu}\d\mu_N=0.
		\]
		Since a translation in direction $\vec{a}$ precisely induces such a normal variation $\spann{\vec{a},\nu}$, and since $\scal=\frac{1}{2}\mathcal{H}^2$ on the Minkowski lightcone, the integral identity for $f_{\vec{a}}$ induced by a Lorentz transformation in direction $\vec{a}$ can be understood as the analogue statement.
	\end{bem}
	\begin{proof}
		We have already seen that $f_{\vec{a}}$ has mean value $0$, so it suffices to show that the second integral vanishes as well. As frequently observed above, a M\"obius transformation $\Phi_{\vec{a},t}$ is a conformal diffeomorphism on $\Sbb^2$. Hence, 
		\[
			X_{\vec{a}}:=\frac{\d}{\d t}\Phi_{\vec{a},t}\vert_{t=0}
		\]
		is a conformal Killing vector field on $\Sbb^2$. By \eqref{eq_explicit_fa_2},
		\[
			\int_\Sigma \scal\cdot f_{\vec{a}}\d\mu=-\frac{1}{2}\int_\Sigma -2\Delta f_{\vec{a}}-2\scal f_{\vec{a}}\d\mu=-\frac{1}{2}\int_{\Sigma}J(f_{\vec{a}})\d\mu=-\int_\Sigma \widehat{\gamma}(\widehat{\nabla}\scal,X_{\vec{a}})\d\mu=0,
		\]
		where the last integral vanishes due to the Kazdan--Warner identity. See for example Le \cite{le1} for a proof utilizing the geometry of the Minkowski lightcone.
	\end{proof}
	
	Although any function $f_{\vec{a}}$ has mean zero and can therefore be used as a test function for stability, the strategies implemented in the previous sections highlight that their main utility arises from a set of test functions that enjoy nice pointwise properties when summed up. 
	
	\begin{prop}
		Let $(\Sigma,\gamma)$ be a conformally round surface with conformal factor $\omega$. Then, there exists functions $f_i$, $i=1,2,3$, on $\Sigma$ such that
		\[
		\int_\Sigma f_i\d\mu=0
		\]
		for $i=1,2,3$ and 
		\begin{align*}
			\sum\limits_{i=1}^3f_i^2&=1+\btr{\widehat{\nabla}\ln\omega}^2,\\
			\sum\limits_{i=1}^3\btr{\widehat{\nabla} f_i}^2&=\btr{\widehat{\nabla}^2\ln\omega}^2+\btr{\widehat{\nabla}\ln\omega}^2+\omega^2\scal_\Sigma.
		\end{align*}
		Moreover, the functions $f_i$ span the kernel of $\Delta+\scal_\Sigma$ if and only if $(\Sigma,\gamma)$ has constant curvature.
	\end{prop}
	\begin{proof}
		The proof is building on the above considerations and purely computational in nature. For the convenience of the reader, we give a brief outline:
		
		Consider $f_i:=\frac{1}{2}f_{\partial_i}$, $i=1,2,3$. First, we recall the well-known fact that the Lorentz transformation $L_{\vec{a},t}$ for $t\not=0$ decomposes as  
		\[
		L_{\vec{a},t}=D_{\vec{a}}\circ \Lambda_t\circ D_{\vec{a}}^{-1},
		\]
		where $\Lambda_t$ is a special Lorentz boost in direction $t\partial_3$, and $D_{\vec{a}}$ is the uniquely determined rotation mapping $\partial_3$ to $\vec{a}$ (without any rotation perpendicular to $\vec{a}$). See e.g. \cite[Section 4.2]{wolff_thesis}. In particular, the M\"obius transformation $\Phi_{\vec{a},t}$ decomposes as 
		\[
		\Phi_{\vec{a},t}=\Phi_{D_{\vec{a}}}\circ \Phi_{t}\circ\Phi_{D^{-1}_{\vec{a}}},
		\]
		where $\Phi_t$ is the uniquely determined conformal diffeomorphism related to $\Lambda_t$, and $\Phi_{D}$ denotes the unique isometry on $\Sbb^2$ induced via the canonical embedding of $\Sbb^2$ into $\R^3$ and a rotation $D$ in $\R^3$. In particular, $\Phi_{D^{-1}_{\vec{a}}}=\Phi^{-1}_{D_{\vec{a}}}$.
		
		Introducing angular coordinates $(\theta,\varphi)\in(0,\pi)\times(0,2\pi)$ on $\Sbb^2$, a tedious but straightforward computation yields the following:
		
		\begin{center}
			\begin{tabular}{|c||c|c|c|c|}
				\hline&&&&\\[-1em]
				$\vec{a}$ & $\vec{a}\cdot\vec{x}$ & $X_{\vec{a}}:=\frac{\d}{\d t}\vert_{t=0}\Phi_{\vec{a},t}$ & $\widehat{\nabla}_\theta X_{\vec{a}}$ & $\widehat{\nabla}_\varphi X_{\vec{a}}$\\[0.2em]
				\hline&&&&\\[-0.9em]
				$\partial_1$ & $\cos\varphi\sin\theta$ & $-\cos\varphi\cos\theta\partial_\theta+\frac{\sin\varphi}{\sin\theta}\partial_\varphi$ & $\cos\varphi\sin\theta\partial_\theta$ & $\cos\varphi\sin\theta\partial_\varphi$\\[0.2em]
				\hline&&&&\\[-0.9em]
				$\partial_2$ & $\sin\varphi\sin\theta$ & $-\sin\varphi\cos\theta\partial_\theta-\frac{\cos\varphi}{\sin\theta}\partial_\varphi$ & $\sin\theta\sin\varphi\partial_\theta$ & $\sin\theta\sin\varphi\partial_\varphi$
				\\[0.2em]
				\hline&&&&\\[-0.9em]
				$\partial_3$ & $\cos\theta$ & $\sin\theta\partial_\theta$ & $\cos\theta\partial_\theta$ & $\cos\theta\partial_\varphi$\\[0.2em]
				\hline
			\end{tabular}
		\end{center}
		Note that in particular, $X_i:=X_{\partial_i}$ are conformal Killing vector fields that span the tangent space at all points on $\Sbb^2$ (covered by the chart of angular coordinates). From this and \eqref{eq_explicit_fa_2}, it is easy to infer that the set of functions $f_i$ span the kernel of $\Delta+\scal=-\frac{1}{2}J$ if and only if $\Sigma$ has constant curvature. Finally, the pointwise identities for the set of functions $f_i$ follow from the above identities after another lengthy, but direct computation.
		
	\end{proof}

	\bibliographystyle{plain}
	\bibliography{bib_stability}
	\,\\\\\\
	{\small
	Markus Wolff\newline
	University of Vienna \newline
	Faculty of Mathematics\newline
	Oskar-Morgenstern Platz 1\newline
	1090 Vienna, Austria\newline
	{https://orcid.org/0000-0002-2257-7359}\newline
	markus.wolff@univie.ac.at}
	
	\nopagebreak
\end{document}